\documentclass{amsart}

\makeatletter
\@namedef{subjclassname@2020}{%
  \textup{2020} Mathematics Subject Classification}
\makeatother % Cambiar el año de la MSC, de 2010 a 2020

\usepackage{mathrsfs}
\usepackage{amsthm,amssymb,amsfonts,latexsym,mathtools,thmtools}
\usepackage[T1]{fontenc}
\usepackage{tikz-cd} % Commutative diagrams.
\usepackage{enumitem} % Customization of items.
\usepackage{hyperref} %hyperlinks
\usepackage{hyperref}
\hypersetup{
    colorlinks=true,
    linkcolor=blue,
    filecolor=blue,      
    urlcolor=blue,
    linktocpage=true
}

\newtheorem{Theorem}{Theorem}[section]
\newtheorem{lemma}[Theorem]{Lemma}
\newtheorem{proposition}[Theorem]{Proposition}
\newtheorem{corollary}[Theorem]{Corollary}

\theoremstyle{definition}
\newtheorem{definition}[Theorem]{Definition}
\newtheorem{example}[Theorem]{Example}

\newtheorem{remark}[Theorem]{Remark}

\numberwithin{equation}{section}

\begin{document}

\title[Uniform and couniform dimensions of inverse polynomial modules]{Uniform and couniform dimensions of inverse polynomial modules over skew Ore polynomials}

%    Remove any unused author tags.

\author{Sebasti\'an Higuera}
\address{Universidad Nacional de Colombia - Sede Bogot\'a}
\curraddr{Campus Universitario}
\email{sdhiguerar@unal.edu.co}
\thanks{}

\author{Armando Reyes}
\address{Universidad Nacional de Colombia - Sede Bogot\'a}
\curraddr{Campus Universitario}
\email{mareyesv@unal.edu.co}
\thanks{}

\thanks{The authors were supported by the research fund of Department of Mathematics, Faculty of Science, Universidad Nacional de Colombia - Sede Bogot\'a, Colombia, HERMES CODE 53880.}

\subjclass[2020]{16D10, 16D60, 16D80, 16E45, 16S36, 16S85, 16W50.}

\keywords{Uniform dimension, couniform dimension, inverse polynomial module, skew polynomial ring, skew Ore polynomial, perfect ring, Bass module.}

\date{}

\dedicatory{Dedicated to Adonail Higuera}

\begin{abstract}

In this paper, we study the uniform and couniform dimensions of inverse polynomial modules over skew Ore polynomials. 

\end{abstract}

\maketitle

%\tableofcontents

\section{Introduction}\label{ch0}

Throughout the paper, every ring $R$ is associative (not necessarily commutative) with identity. If $R$ is commutative, then it is denoted by $K$. An important tool to investigate theoretical properties of rings and modules is the {\em uniform dimension}. If $M_R$ is a right module (resp., $_RM$ is a left module), its {\em right uniform dimension} (resp., {\em left uniform dimension}) is denoted by ${\rm rudim}(M_R)$ (resp., ${\rm ludim}(_RM)$).

%\cite[theorem 2.6]{Shock1972} \cite[Corollary 4]{Grzeszczuk1988} \cite[theorem 15]{Quinn1988}

Shock \cite{Shock1972} proved that if $R$ has finite left uniform dimension, then the left uniform dimensions of the commutative polynomial ring $R[x]$ and $R$ are equal \cite[Theorem 2.6]{Shock1972}. In the setting of the {skew polynomial rings} defined by Ore \cite{Ore1931, Ore1933} as the ring $R[x;\sigma, \delta]$ where $\sigma$ is an endomorphism of $R$ and $\delta$ is a $\sigma$-derivation of $R$, Grzeszczuk \cite{Grzeszczuk1988} showed that $R[x; \delta]_{R[x;\delta]}$ and $R_R$ have the same uniform dimension if $R$ is right nonsingular or if $R$ is a $\mathbb{Q}$-algebra satisfying the descending chain condition on right annihilators \cite[Corollary 4]{Grzeszczuk1988}. Quinn \cite{Quinn1988} proved that if $R$ is a $\mathbb{Q}$-algebra and $\delta$ is locally nilpotent (a derivation $\delta$ is called {\em locally nilpotent} if for all $r \in R$ there exists $n(r)\ge 1$ such that $\delta^{n(r)}(r) = 0$ \cite[p. 11]{Freudenburg2006}), then $R[x; \delta]_{R[x;\delta]}$ and $R_R$ have the same uniform dimension \cite[Theorem 15]{Quinn1988}. Following Matczuk \cite{Matczuk1995}, if $R$ is semiprime left Goldie and $\sigma$ is an automorphism of $R$, then the left uniform dimensions of $R[x;\sigma, \delta]$ and $R$ are equal. Leroy and Matczuk \cite{LeroyMatczuk2005} generalized this result to the case where $\sigma$ is an injective endomorphism of $R$. Later, they investigated the uniform dimension of induced modules over $R[x;\sigma,\delta]$, and proved that $M_R \otimes_R R[x;\sigma,\delta]$ and $M_R$ have the same uniform dimension \cite[Lemma 4.9]{LeroyMatczuk2004}. 

Annin \cite{AnninPhD2002} studied the left uniform dimension of the {\em polynomial module} $M[x]_S$ which consists of all polynomials of the form $m(x)=m_0 + \cdots + m_kx^{k}$ with $m_i \in M_R$ for all $0 \le i\le k$, and $S:=R[x;\sigma]$ where $\sigma$ is an automorphism of $R$. Under certain compatibility conditions, Annin proved that $M_R$ and $M[x]_{S}$ have the same right uniform dimension \cite[Theorem 4.15]{AnninPhD2002}. In addition, he also investigated the couniform dimension of the {\em inverse polynomial module} $M[x^{-1}]_S$ formed by the polynomials $m(x)=m_0 + m_1x^{-1}+ \cdots + m_kx^{-k}$ with $m_i \in M_R$ for all $0 \le i\le k$. Annin showed that ${\rm rudim}(M_R)= {\rm rudim}(M[x^{-1}]_S)$ \cite[Theorem 4.7]{AnninPhD2002}.

%In the setting of the {\em skew PBW extensions} introduced by Gallego and Lezama \cite{GallegoLezama2011} as a generalization of the Poincar\'e-Birkhoff-Witt extensions considered by Bell and Goodearl \cite{BellGoodearl1988} and the skew polynomial rings of injective type, Reyes \cite{Reyes2014} studied the left uniform dimension of this family of noncommutative algebras, and proved that $R$ and a skew PBW extension built over $R$ have the same left uniform dimension \cite[theorem 4.11]{Reyes2014}. Higuera et al. \cite{HigueraRamirezReyes2024} characterized the essential modules and the uniform dimension of induced modules over skew PBW extensions. they showed that $M \langle X \rangle_A:= M \otimes_R A$ and $M_R$ have the same left uniform dimension, where $A$ is a bijective skew PBW extension over $R$ \cite[theorem 3.16]{HigueraRamirezReyes2024}.

Varadarajan \cite{Varadarajan1979} introduced the concept of {\em couniform dimension} as a dual theory of uniform dimension and denoted it by ${\rm corank}(M_R)$. Sarath and Varadarajan \cite{Sarathetal1979} showed that if $M_R$ has finite couniform dimension (that is, ${\rm corank}(M_R)< \infty$), then $M_R/J(M_R)$ is semisimple Artinian, where $J(M_R)$ denotes the {\em Jacobson radical} of $M_R$ \cite[Corollary 1.11]{Sarathetal1979}. Under certain conditions, they proved that $M_R/J(M_R)$ being semisimple Artinian is sufficient for $M_R$ to have finite couniform dimension \cite[Theorem 1.13]{Sarathetal1979}, and showed that $R_R$ has finite couniform dimension when $R_R$ right finite couniform dimension \cite[Corollary 1.14]{Sarathetal1979}. We recall that $R$ is called {\em right perfect} if $R/J(R_R)$ is semisimple and for every sequence $\{a_n\ | \ n \in \mathbb{N}\} \subseteq J(R)$, there exists $k \in \mathbb{N}$ such that $a_ka_{k-1}\cdots a_1 = 0$. Annin \cite{Annin2005} studied the couniform dimension of $M[x^{-1}]_S$ and proved that $M_R$ and $M[x^{-1}]_S$ have the same couniform dimension when $R$ is right perfect \cite[Theorem 2.10]{Annin2005}. 

Cohn \cite{Cohn1961} introduced the {\em skew Ore polynomials of higher order} as a generalization of the skew polynomial rings considering the relation $xr := \Psi_1(r)x + \Psi_2(r)x^2 + \dotsb$ for all $r \in R$, where the $\Psi$'s are endomorphisms of $R$. Following Cohn's ideas, Smits \cite{Smits1968} introduced the ring of skew Ore polynomials of higher order over a division ring $D$ and commutation rule defined by
\begin{align}
    xr:= r_1x + \cdots + r_kx^k,\ \text{for all}\ r\in R\ \text{and}\ k \ge 1.\label{eq:smits}  
\end{align}
         
The relation (\ref{eq:smits}) induces a family of endomorphisms $\delta_1, \ldots, \delta_k$ of the group $(D,+)$ with $\delta_i(r):=r_i$ for every $1\le i\le k$ \cite[p. 211]{Smits1968}. Smits proved that if $\{\delta_2, \ldots, \delta_k \}$ is a set of left $D$-independient endomorphisms (i.e., if $c_2\delta_2(r)+ \cdots + c_k\delta_k(r)=0$ for all $r \in D$ then $c_i=0$ for all $2 \le i \le k$ \cite[p. 212]{Smits1968}), then $\delta_1$ is a injective endomorphism \cite[p. 213]{Smits1968}. There exist some algebras such as Clifford algebras, Weyl-Heisenberg algebras, and Sklyanin algebras, in which this commutation relation is not sufficient to define the noncommutative structure of the algebras since a free non-zero term $\Psi_0$ is required. Maksimov \cite{Maksimov2000} considered the skew Ore polynomials of higher order with free non-zero term $\Psi_0(r)$ where $\Psi_0$ satisfies the relation $\Psi_0(rs) = \Psi_0(r)s + \Psi_1(r)\Psi_0(s) + \Psi_2(r)\Psi_0^{2}(s) + \dotsb$, for every $r, s \in R$. Later, Golovashkin and Maksimov \cite{Golovashkinetal2005} introduced the algebras $Q(a,b,c)$ over a field $\Bbbk$ of characteristic zero with two generators $x$ and $y$, subject to the quadratic relations $yx = ax^2 + bxy + cy^2$, where $a,b,c \in \Bbbk$. If $\{x^my^n\}$ forms a basis of $Q(a,b,c)$ for all $m,n \ge 0$, then the ring generated by the quadratic relation is an algebra of skew Ore polynomials, and it can be defined by a system of linear mappings $\delta_0,\ldots,\delta_k$ of $\Bbbk[x]$ into itself such that for any $p(x) \in \Bbbk[x]$, $yp(x) = \delta_0(p(x)) + \delta_1(p(x))y + \cdots + \delta_k(p(x))y^k$, for some $k \in \mathbb{N}$. 

Motivated by Annin's research \cite{Annin2005} about the couniform dimension of $M[x^{-1}]_S$ and the importance of the algebras of skew Ore polynomials of higher order, in this paper we study this dimension for inverse polynomial modules over skew Ore polynomials considered by the authors in \cite{HigueraReyes2023a}. Since some of its ring-theoretical, homological and combinatorial properties have been investigated recently (e.g., \cite{ChaconReyes2023, Ninoetal2024, NinoReyes2023} and references therein), this article can be considered as a contribution to the research on these objects.

The paper is organized as follows. Section \ref{Definitions} recalls some definitions and key results about skew Ore polynomials and completely $(\sigma,\delta)$-compatible modules. In Section \ref{uniformdimension} we prove that if $N_R$ is an essential submodule of $M_R$ then $N[x^{-1}]_A$ is an essential submodule of $M[x^{-1}]_A$, where $A$ is a skew Ore polynomial ring and $M_R$ is completely $(\sigma,\delta)$-compatible (Lemma \ref{essentialmodule}), and if $N_R$ is a uniforme submodule of $M_R$, then $N[x^{-1}]_A$ is an uniform submodule of $M[x^{-1}]_A$ (Lemma \ref{uniformmodule}). We also show that if $M_R$ is completely $(\sigma,\delta)$-compatible, then $M_R$ and $M[x^{-1}]_A$ have the same right uniform dimension (Theorem \ref{theoremuniform}). In Section \ref{couniformdimension} we investigate the hollow modules of $M[x^{-1}]_A$ (Lemma \ref{AnninLemma2.9}) and show that the couniform dimensions of $M_R$ and $M[x^{-1}]_A$ are equal when $R$ is right perfect (Theorem \ref{Annintheorem2.10}). As expected, our results extend those above corresponding to skew
polynomial rings of automorphism type presented by Annin \cite{Annin2005}. Finally, we say some words about a future research. 

The symbols $\mathbb{N}$, $\mathbb{Z}$, $\mathbb{R}$, and $\mathbb{C}$ denote the set of natural numbers including zero, the ring of integer numbers, and the fields of real and complex numbers, respectively. The term module will always mean right module unless stated otherwise. The symbol $\Bbbk$ denotes a field and $\Bbbk^{*} := \Bbbk\ \backslash\ \{0\}$.

\section{Preliminaries}\label{Definitions}

If $\sigma$ is an endomorphism of $R$, then a map $\delta : R \rightarrow R$ is called a {\em $\sigma$-derivation} of $R$ if it is additive and satisfies that $\displaystyle \delta(r s) = \sigma(r )\delta(s)+\delta(r )s$ for every  $r,s \in R$ \cite[p. 26]{GoodearlWarfield2004}. According to Ore \cite{Ore1931, Ore1933}, the {\em skew polynomial ring} (or {\em Ore extension}) is defined as the ring $R[x;\sigma,\delta]$ generated by $R$ and the indeterminate $x$ such that it is a free left $R$-module with basis $\left \{x^k\ | \ k \in \mathbb{N} \right \}$ subject to the relation $xr := \sigma(r)x + \delta(r)$ for all $r \in R$ \cite[p. 34]{GoodearlWarfield2004}. We present the skew Ore polynomial rings introduced by the authors \cite{HigueraReyes2023a}.

\begin{definition}[{\cite[Definition 2.2]{HigueraReyes2023a}}]\label{importantdefinition}
    If $\sigma$ is an automorphism of $R$ and $\delta$ is a locally nilpotent $\sigma$-derivation of $R$, we define the {\em skew Ore polynomial ring} $A:=R(x; \sigma,\delta)$ which consists of the uniquely representable elements $r_0 + r_1x + \cdots + r_kx^k$ where $r_i \in R$ and $k \in \mathbb{N}$, with the commutation rule $xr:= \sigma(r)x + x\delta(r)x$, for all $r \in R$.
\end{definition}

According to Definition \ref{importantdefinition}, if $r\in R$ and $\delta^{n(r)}(r) = 0$ for some $n(r)\ge 1$, then
\begin{align}
    xr= \sigma(r)x + \sigma\delta(r)x^2 + \cdots + \sigma\delta^{n(r)-1}(r)x^{n(r)}.\label{eq:our}
\end{align}

If we define the endomorphisms $\Psi_i:=\sigma\delta^{i-1}$ for all $i\ge 1$ and $\Psi_0:=0$, then $A$ is a skew Ore polynomial of higher order in the sense of Cohn \cite{Cohn1961}.

\begin{example}[{\cite[Example 2.3]{HigueraReyes2023a}}]\label{ExampleskewOre} We present some examples of skew Ore polynomial rings.
    \begin{enumerate} 
        \item If $\delta=0$ then $xr=\sigma(r)x$, and thus $R(x;\sigma)=R[x;\sigma]$ is the skew polynomial ring where $\sigma$ is an automorphism of $R$.
        \item The {\em quantum plane} $\Bbbk_q[x,y]$ is the algebra generated by $x, y$ over $\Bbbk$ subject to the commutation rule $xy=qyx$ with $q \in \Bbbk^*$ and $q\neq 1$. We note that $\Bbbk_q[x,y] \cong \Bbbk[y](x;\sigma)$, where $\sigma(y):=qy$ is an automorphism of $\Bbbk[y]$.
        \item The {\em Jordan plane} $\mathcal{J}(\Bbbk)$ defined by Jordan \cite{Jordan2001} is the free algebra generated by the indeterminates $x, y$ over $\Bbbk$ and the relation $yx = xy + y^2$. This algebra can be written as the skew polynomial ring $\Bbbk[y][x;\delta]$ with $\delta(y):=-y^2$. On the other hand, notice that $\delta(x)= 1$ is a locally nilpotent derivation of $\Bbbk[x]$, and thus the Jordan plane also can be seen as $\Bbbk[x](y;\delta)$.
        \item D\'iaz and Pariguan \cite{DiazPariguan2009} introduced the {\em $q$-meromorphic Weyl algebra} $MW_q$ as the algebra generated by $x,y$ over $\mathbb{C}$, and defining relation $yx=qxy + x^2$, for $0 < q < 1$. Lopes \cite{Lopes2023} showed that using the generator $Y =y+(q-1)^{-1}x$ instead of $y$, it follows that $Yx = qxY$, and thus the $q$-meromorphic Weyl algebra is the quantum plane $\mathbb{C}_q[x,y]$ \cite[Example 3.1]{Lopes2023}. 
        \item Consider the algebra $Q(0,b,c)$ defined by Golovashkin and Maksimov \cite{Golovashkinetal2005} with $a=0$. It is straightforward to see that $\sigma(x)=bx$ is an automorphism of $\Bbbk[x]$ with $b \neq 0$, $\delta(x)=c$ is a locally nilpotent $\sigma$-derivation of $\Bbbk[x]$ and so $Q(0,b,c)$ can be interpreted as $A=\Bbbk[x](y;\sigma,\delta)$.
        \item If $\delta_1$ is an automorphism of $D$ and $\{\delta_2, \ldots, \delta_k \}$ is a set of left $D$-independient endomorphism, then $\delta:=\delta_1^{-1} \delta_2$ is a $\delta_1$-derivation of $D$, $\delta_{i+1}(r)= \delta_1\delta^{i}(r)$, and $\delta^{k}(r)=0$ for all $r \in D$ \cite[p. 214]{Smits1968}, and thus (\ref{eq:smits}) coincides with (\ref{eq:our}). In this way, the skew Ore polynomial rings of higher order defined by Smits can be seen as $D(x;\delta_1,\delta)$.
    \end{enumerate}
\end{example}

We recall that a multiplicative subset $X$ of $R$ satisfies {\em left Ore condition} if $Xr \cap Rx \neq \emptyset$, for all $r \in R$ and $x \in X$. If this is the case, then $X$ is called a {\em left Ore set}. The authors \cite{HigueraReyes2023a} proved that $X=\{x^k\ | \ k \ge 0 \}$ is a left Ore set of the algebra $A$ \cite[Proposition 2.3]{HigueraReyes2023a}. In this way, we can localize $A$ by $X$ and denote it by $X^{-1}A$. It is not hard to see that $x^{-1}$ satisfies the relation $x^{-1}r:= \sigma'(r)x^{-1} + \delta'(r)$, for all $r \in R$ with $\sigma'(r):=\sigma^{-1}(r)$ and $\delta'(r):=-\delta\sigma^{-1}(r)$. 

Dumas \cite{Dumas1991} studied the field of fractions of $D[x;\sigma,\delta]$ where $\sigma$ is an automorphism of $D$ and stated that one technique for this purpose is to consider it as a subfield of a certain field of series \cite[p. 193]{Dumas1991}: if $Q$ is the field of fractions of $D[x;\sigma,\delta]$, then $Q$ is a subfield of the {\em field of series of Laurent} $D((x^{-1};\sigma^{-1}, -\delta\sigma^{-1}))$ whose elements are of the form $r_{-k}x^{-k}+\cdots + r_{-1}x^{-1} + r_0 + r_1x+ \cdots$ for some $k \in \mathbb{N}$, and satisfy the commutation rules given by
\begin{align*}
    xr&:= \sigma(r)x + \sigma\delta(r)x^2 + \cdots = \sigma(r)x +x\delta(r)x,\ \text{and}\\
     x^{-1}r&:= \sigma'(r)x^{-1} + \delta'(r),\ \text{for all}\ r\in R.
\end{align*}

By \cite[Proposition 2.3]{HigueraReyes2023a}, if $\sigma$ is an automorphism of $D$ and $\delta$ is a locally nilpotent $\sigma$-derivation of $D$ then $X^{-1}A \subseteq D((x^{-1};\sigma^{-1}, -\delta\sigma^{-1}))$.

%\begin{remark} [{\cite[Remark 2.4]{HigueraReyes2023a}}]

Following \cite[Remark 2.4]{HigueraReyes2023a}, if $\sigma$ is an automorphism of $R$ and $\delta$ is a $\sigma$-derivation of $R$, we denote by $f_j^i$ the endomorphism of $R$ which is the sum of all possible words in $\sigma',\delta'$ built with $i$ letters $\sigma'$ and $j-i$ letters $\delta'$, for $i \leq j$. In particular, $f_0^0 = 1$, $f_j^j = \sigma'^{j}$, $f_j^0 = \delta'^j$ and $f_j^{j-1} = \sigma'^{j-1}\delta' +\sigma'^{j-2}\delta'\sigma' +\cdots+\delta'\sigma'^{j-1}$; if $\delta\sigma=\sigma\delta$ then $f_j^{i}= \binom{j}{i} \sigma'^i\delta'^{j-i}$. If $r \in R$ and $k \in\mathbb{N}$, then the following formula holds:
\begin{align}
    \displaystyle x^{-k}r =\sum_{i=0}^kf_{k}^{i}(r)x^{-i}.\label{relacion2}
\end{align}

In addition, if $r,s \in R$ and $k, k' \in \mathbb{N}$ then
\begin{align}
    \displaystyle (rx^{-k})(sx^{-k'}) = \sum_{i=0}^{k}rf_k^i(s)x^{-(k+k')}. \label{relacion3}
\end{align}
%\end{remark}

Considering the usual addition of polynomials and the product induced by (\ref{relacion2}) and (\ref{relacion3}), the authors defined the ring of polynomials in the indeterminate $x^{-1}$ with coefficients in $R$ and denote it by $R[x^{-1}]$. The {\em inverse polynomial module} $M[x^{-1}]_R$ is defined as the set of all the polynomials of the form $f(x)=m_0 + \cdots + m_kx^{-k}$ with $m_i \in M_R$ for all $1\le i \le n$, together with the usual addition of polynomials and the product given by (\ref{relacion2}) as follows:
\begin{align}
    \displaystyle mx^{-k}r :=\sum_{i=0}^kmf_{k}^{i}(r)x^{-i},\ \text{for all}\ m\in M_R\ \text{and}\ r\in R. 
\end{align}

\begin{remark}[{\cite[Remark 2.5]{HigueraReyes2023a}}]
If $m(x)=m_0 + m_1x^{-1} + \cdots + m_kx^{-k} \in M[x^{-1}]_R$, the \textit{leading monomial} of $m(x)$ is denoted as ${\rm lm}(m(x))=x^{-k}$, the \textit{leading coefficient} of $m(x)$ by ${\rm lc}(m(x))=m_k$, and the \textit{leading term} of $m(x)$ as ${\rm lt}(m(x))=m_kx^{-k}$. We define the {\em negative degree} of $x^{-k}$ by $\deg(x^{-k}):= -k$ for any $k \in \mathbb{N}$, and the {\em negative degree} of $m(x)$ is given by $\deg(m(x))={\rm max}\{\deg(x^{-i})\}_{i=0}^k$ for all $m(x)\in M[x^{-1}]_R$. For any $m(x) \in M[x^{-1}]_R$, we denote by $C_{m}$ the set of the coefficients of $m(x)$.
\end{remark}

%\subsection{Completely $(\sigma,\delta)$-compatible rings} 
According to Annin \cite{Annin2004} (c.f. Hashemi and Moussavi \cite{HashemiMoussavi2005}), if $\sigma$ is an endomorphism of $R$ and $\delta$ is a $\sigma$-derivation of $R$, then $M_R$ is called $\sigma$-{\em compatible} if for each $m\in M_R$ and $r\in R$, we have that $mr = 0$ if and only if $m\sigma(r)=0$; $M_R$ is $\delta$-{\em compatible} if for $m \in M_R$ and $r\in R$, $mr = 0$ implies that $m\delta(r)=0$; if $M_R$ is both $\sigma$-compatible and $\delta$-compatible, then $M_R$ is called a ($\sigma,\delta$)-{\em compatible module} \cite[Definition 2.1]{Annin2004}. Annin introduced a stronger notion of compatibility to study the attached prime ideals of the inverse polynomial module $M[x^{-1}]_S$. Following \cite[Definition 1.4]{Annin2011}, $M_R$ is called {\it completely $\sigma$-compatible} if $(M/N)_R$ is $\sigma$-compatible, for every submodule $N_R$ of $M_R$. The authors \cite{HigueraReyes2023a} introduced the following definition with the aim of studying the attached prime ideals of $M[x^{-1}]_A$ (see Section \ref{uniformdimension}).

\begin{definition}[{\cite[Definition 3.1]{HigueraReyes2023a}}] If $\sigma$ is an endomorphism of $R$ and $\delta$ is a $\sigma$-derivation of $R$, then $M_R$ is called {\em completely $\sigma$-compatible} if for each submodule $N_R$ of $M_R$, we have that $(M/N)_R$ is $\sigma$-compatible; $M_R$ is {\it completely $\delta$-compatible} if for each submodule $N_R$ of $M_R$, we obtain that $(M/N)_R$ is $\delta$-compatible; $M_R$ is {\it completely $(\sigma,\delta)$-compatible} if it is both completely $\sigma$-compatible and $\delta$-compatible.
\end{definition}

%Let us see some examples of completely $(\sigma,\delta)$-compatible modules. 
\begin{example} [{\cite[Example 3.2]{HigueraReyes2023a}}] 
\begin{itemize}
   \item[{\rm (1)}] If $M_R$ is simple and $(\sigma,\delta)$-compatible, then $M_R$ is completely $(\sigma,\delta)$-compatible. 
    \item[{\rm (2)}] Let $K$ be a local ring with maximal ideal $\mathfrak{m}$ and $\sigma$ an automorphism of $K$. Annin \cite{AnninPhD2002} proved that $M_K := K/\mathfrak{m}$ is a $\sigma$-compatible module, and since $M_K$ is simple it follows that $M_K$ is completely $\sigma$-compatible \cite[Example 3.35]{AnninPhD2002}. If $\delta$ is a $\sigma$-derivation of $K$ such that $\delta(r) \in \mathfrak{m}$ for every $r \in \mathfrak{m}$, then $M_K$ is completely $\delta$-compatible. Indeed, if $\overline{0} \neq \overline{s} \in M_K$ and $r \in K$ satisfy that $\overline{s}r = 0$ then $sr \in \mathfrak{m}$, and since $s \notin \mathfrak{m}$ we obtain that $r \in \mathfrak{m}$. If $\delta(r) \in \mathfrak{m}$ for all $r \in \mathfrak{m}$, it follows that $s\delta(r) \in \mathfrak{m}$ and so $\overline{s}\delta(r) = 0$. Therefore, $M_K$ is $\delta$-compatible and thus $M_K$ is completely $\delta$-compatible. 
\end{itemize}
\end{example}

%The following proposition presents some properties of completely $(\sigma,\delta)$-compatible modules. These properties are required throughout the paper.

%\begin{proposition}[{\cite[Proposition 3.3]{HigueraReyes2023a}}]\label{Propertycompletely}
%If $M_R$ is completely $(\sigma, \delta)$-compatible and $N_R$ is a submodule of $M_R$, then the following assertions hold:
%\begin{enumerate}
 %   \item[{\rm (1)}] If $ma \in N_R$ then $m\sigma^i(a), m\delta^{j}(a) \in N_R$ for each $i,j \in \mathbb{N}$.
  %  \item[{\rm (2)}] If $mab \in N_R$ then $m\sigma(\delta^{j}(a))\delta(b), m\sigma^{i}(\delta(a))\delta^{j}(b) \in N_R$ for all $i,j \in \mathbb{N}$. In particular, $ma\delta^{j}(b), m\delta^{j}(a)b \in N_R$ for all $j \in \mathbb{N}$.
  %  \item[{\rm (3)}] If $mab \in N_R$ or $m\sigma(a)b \in N_R$ then $m\delta(a)b \in N_R$.
%\end{enumerate}
%\end{proposition}

We recall some properties of completely $(\sigma,\delta)$-compatible modules. 

\begin{proposition} \label{Propertycompletely}
    \begin{itemize} 
        \item[{\rm (a)}] \cite[Proposition 3.3]{HigueraReyes2023a} If $M_R$ is completely $(\sigma, \delta)$-compatible and $N_R$ is a submodule of $M_R$, then the following assertions hold:
\begin{enumerate}
    \item[{\rm (1)}] If $mr \in N_R$ then $m\sigma^i(r), m\delta^{j}(r) \in N_R$ for each $i,j \in \mathbb{N}$.
    \item[{\rm (2)}] If $mrr' \in N_R$ then $m\sigma(\delta^{j}(r))\delta(r'), m\sigma^{i}(\delta(r))\delta^{j}(r') \in N_R$ for all $i,j \in \mathbb{N}$. In particular, $mr\delta^{j}(r'), m\delta^{j}(r)r' \in N_R$ for all $j \in \mathbb{N}$.
    \item[{\rm (3)}] If $mrr' \in N_R$ or $m\sigma(r)r' \in N_R$ then $m\delta(r)r' \in N_R$.
\end{enumerate}
        \item[{\rm (b)}] \cite[Proposition 3.4]{HigueraReyes2023a} If $\sigma$ is an endomorphism of $R$, $\delta$ is a $\sigma$-derivation of $R$ and $M_R$ is completely $(\sigma,\delta)$-compatible, then
  \begin{itemize}
    \item[\rm (1)] $M_R$ is a $(\sigma,\delta)$-compatible module. 
    \item[\rm (2)] $(M/N)_R$ is completely $(\sigma, \delta)$-compatible for all submodule $N_R$ of $M_R$.
\end{itemize}
     \item[{\rm (c)}] \cite[Proposition 3.5]{HigueraReyes2023a} If $\sigma$ is bijective and $M_R$ is completely $(\sigma,\delta)$-compatible, then $M_R$ is a completely $(\sigma',\delta')$-compatible module.
    \end{itemize}
\end{proposition}

\section{Uniform dimension of inverse polynomial modules}\label{uniformdimension}  
In this section, we study the uniform dimension of the inverse polynomial module $M[x^{-1}]_A$. We define a structure of $A$-module for $M[x^{-1}]$ as follows: 
    \begin{align}
        mx^{-1}r &:= m\sigma'(r)x^{-1} + m\delta'(r),\ \text{for all}\ r \in R\ \text{and}\ m \in M_R,\ \text{and} \label{eqn:(3.1)}\\
         x^{-i}x^{j}&:=x^{-i+j}\ \text{if}\ j \le i\ \text{and}\ 0\ \text{otherwise} \label{eqn:(3.2)}.
    \end{align}

It follows from (\ref{eqn:(3.1)}) and (\ref{eqn:(3.2)}) that if $\delta:= 0$ then $mx^{-i}rx^j:=m\sigma'^{i}(r)x^{-i+j}$ for all $r \in R$ and $i,j \in \mathbb{N}$ with $j \le i$. This coincides with $M[x^{-1}]_S$ \cite[Remark 3.5]{HigueraReyes2023a}. 

A submodule $N_R$ of $M_R$ is {\em essential} if $mR \cap N_R \neq 0$ for all non-zero element  $m \in M_R$, i.e. there exists $r\in R$ such that $mr \in N_R$ \cite[Definition 3.26]{Lam1998}. Next, we investigate the property of being essential and its passage from $M_R$ to $M[x^{-1}]_A$.

\begin{lemma}\label{essentialmodule}
If $M_R$ is a completely $(\sigma,\delta)$-compatible module and $N_R$ is an essential submodule of $M_R$, then $N[x^{-1}]_A$ is an essential submodule of $M[x^{-1}]_A$.
\end{lemma}
\begin{proof}
    Suppose that $N_R$ is an essential submodule of $M_R$. Let $0\neq m(x) \in M[x^{-1}]_A$ with degree $k$ and leading coefficient $m_{k}$. Since $N_R$ is essential, there exists $r\in R$ such that $0\neq m_kr \in N_R$. Notice that $m(x)(x^{k-1}\sigma(r))=m_krx^{-1} - m_k\delta(r)$. Thus, if $m_kr \in N_R$ and $M_R$ is completely $(\sigma,\delta)$-compatible, then $m_k\delta(r)\in N_R$ and hence $m(x)x^{k-1}\sigma(r) \in N[x^{-1}]_A$ proving that $N[x^{-1}]_A$ is essential in $M[x^{-1}]_A$.
\end{proof}

If every submodule $N_R$ of $M_R$ is essential, then $M_R$ is called {\em uniform}. Following \cite[p. 84]{Lam1998}, $M_R$ is {uniform} if the intersection of any two non-zero submodules of $M_R$ is non-zero. The following lemma studies the property uniform and its passage from $M_R$ to $M[x^{-1}]_A$.

\begin{lemma}\label{uniformmodule}
If $N_R$ is a uniform submodule of $M_R$ then $N[x^{-1}]_A$ is a uniform submodule of $M[x^{-1}]_A$.
\end{lemma}
\begin{proof}
   Assume that $N_R$ is a uniform submodule of $M_R$, and let $n(x), n'(x)$ be two non-zero polynomials of $N[x^{-1}]_A$ with leading coefficients $m$ and $m'$, respectively. If $N_R$ is uniform then $0 \neq mR \cap m'R \subseteq n(x)A \cap n'(x)A$, and therefore $N[x^{-1}]_A$ is a uniform submodule of $M[x^{-1}]_A$.
\end{proof}

If there exist uniform submodules $U_1, \ldots , U_n$ of $M_R$ such that $U_1 \oplus \cdots \oplus U_n$ is an essential submodule of $M_R$, then $M_R$ has {\em finite uniform dimension} and it is denoted by ${\rm rudim}(M_R) = n$ \cite[Definition 6.2]{Lam1998}. According to \cite[Proposition 6.4]{Lam1998}, $M_R$ has infinite uniform dimension if and only if $M_R$ contains an infinite direct sum of non-zero submodules of $M_R$.

The following theorem shows that $M_R$ and $M[x^{-1}]_A$ have the same right uniform dimension and generalizes \cite[Theorem 4.7]{AnninPhD2002}.

\begin{Theorem}\label{theoremuniform}
If $M_R$ is completely $(\sigma,\delta)$-compatible, then
\[
{\rm rudim}(M[x^{-1}]_A)={\rm rudim}(M_R).
\]
\end{Theorem}
\begin{proof}
If ${\rm rudim}(M_R) = n$, then there exist uniform submodules $N_1, \ldots, N_n$ of $M_R$ such that $N_1 \oplus \cdots \oplus N_n$ is an essential submodule of $M_R$. By Lemmas \ref{essentialmodule} and \ref{uniformmodule}, it follows that $N_1[x^{-1}],\ldots, N_n[x^{-1}]$ are uniform submodules of $M[x^{-1}]_A$, and $N_1[x^{-1}] \oplus \cdots \oplus N_n[x^{-1}]$ is essential in $M[x^{-1}]_A$, proving that ${\rm rudim}(M[x^{-1}]_A)=n$.

If ${\rm rudim}(M_R) = \infty$, there exist non-zero submodules $N_1, N_2, \ldots$ of $M_R$ such that $N_1 \oplus N_2 \oplus \cdots$ is a submodule of $M_R$. Thus $N_i[x^{-1}]_A$ is a non-zero submodule of $ M[x^{-1}]_A$ for all $i \ge 1$ and $N_1[x^{-1}]\oplus N_2[x^{-1}]\oplus \cdots$ is a submodule of $M[x^{-1}]_A$, which implies that ${\rm rudim}(M[x^{-1}]_A)= \infty$. Therefore ${\rm rudim}(M[x^{-1}]_A)={\rm rudim}(M_R)$.
\end{proof}

\section{Couniform dimension}\label{couniformdimension}
In this section, we investigate the couniform dimension of the inverse polynomial module $M[x^{-1}]_A$.

\begin{definition}[{\cite[Definition 1.8]{Varadarajan1979}}] \label{DefinitionCouniform} The {\em couniform dimension} of $M_R$ is given by
\[
{\rm corank}(M_R) = {\rm sup} \{k \ | \ M_R\ \text{surjects onto a direct sum of}\ k \ \text{non-zero modules} \}.
\] 
In particular, ${\rm corank}(\{0\}) = 0$.
\end{definition}

A submodule $N_R$ of $M_R$ is {\em small} if $N_R' + N_R = M_R$ implies $N_R' = M_R$ for every submodule $N_R'$ of $M_R$; if $N_R$ is a small submodule of $M_R$, we write $N_R \subseteq_s M_R$ \cite[p. 74]{Lam1998}. If every proper submodule $N_R$ of $M_R$ is small, then $M_R$ is called a {\em Hollow module} \cite[Definition 1.10]{Varadarajan1979}. $M_R$ is hollow if the sum of any two proper submodules remains proper \cite[Definition 1.4]{Annin2005}. Varadarajan proved that $M_R$ is hollow if and only if ${\rm corank}(M_R)=1$ \cite[Proposition 1.11]{Varadarajan1979}. 

\begin{proposition} 
    \begin{itemize}
        \item[{\rm (a)}] \cite[Proposition 1.3]{Annin2005}\label{proposition1.3couniform} The following assertions hold:
        \begin{itemize}
    \item[\rm (1)]  If $N_R$ is a submodule of $M_R$ then ${\rm corank}((M/N)_R) \le {\rm corank}(M_R)$.
    \item[\rm (2)] If $N_R$ is a small submodule of $M_R$ then ${\rm corank}(M_R) = {\rm corank}((M/N)_R)$. The converse holds if ${\rm corank}(M_R) < \infty$.
   \item[\rm (3)] If $M_1, M_2, \ldots, M_n$ are $R$-modules then 
   \[
   {\rm corank}\left ( \bigoplus\limits_{i=1}^n M_i \right) = \sum_{i=0}^n {\rm corank}(M_i).
   \]
\end{itemize}

\item[{\rm (b)}] \cite[Theorem 1.20]{Varadarajan1979}\label{Proposition1.5Couniform}  ${\rm corank}(M_R)<\infty$ if and only if there exist $H_1, \ldots, H_k$ hollow $R$-modules and a surjective homomorphism $\varphi : M \rightarrow  H_1 \oplus \cdots \oplus H_k$ such that ${\rm ker}\ \varphi \subseteq_s M$.
    \end{itemize}
\end{proposition}

$M_R$ is called a {\em Bass module} if every proper submodule is contained in a maximal submodule of $M_R$ \cite[p. 205]{Faith1995}. According to \cite[Exercise 24.7]{Lam1991}, $R$ is right perfect if and only if $R/J(R)$ is semisimple and every non-zero module $M_R$ has a maximal submodule. If $P_R$ is a submodule of $M[x^{-1}]_R$, we set $P_{k} := \{m \in M\ |\  mx^{-k} \in P \}$ for each $k \in \mathbb{N}$ and denote by $\langle P_k \rangle$ the submodule of $M_R$ generated by $P_k$. The following lemma extends \cite[Lemma 2.8]{Annin2005}.

\begin{lemma} \label{AnninLemma2.8}
If $M[x^{-1}]_R$ is a right Bass module, then $N_R \subseteq_s M_R$ if and only if $N[x^{-1}]_A \subseteq_s M[x^{-1}]_A$.
\end{lemma}
\begin{proof}
If $N_R$ is not a small submodule of $M_R$, there exists a submodule $L_R$ of $M_R$ such that $M_R = L_R + N_R$ whence $M[x^{-1}]_A = L[x^{-1}]_A + N[x^{-1}]_A$. Thus $N[x^{-1}]_A$ is not a small submodule of $M[x^{-1}]_A$. 

If $N[x^{-1}]_A$ is not a small submodule of $M[x^{-1}]_A$, then there exists a submodule $Q_A$ such that $Q_A \subsetneq M[x^{-1}]_A$ with $M[x^{-1}]_A = Q_A + N[x^{-1}]_A$. Since $M_R$ is a Bass module, there exists a maximal submodule $P_R$ of $M[x^{-1}]_R$ such that $Q_R \subseteq P_R$ whence $M[x^{-1}]_R = P_R + N[x^{-1}]_R$. It is straightforward to see that $N[x^{-1}]_R \nsubseteq P_R$, and so there exists $nx^{-k} \in N[x^{-1}]_R$ such that $nx^{-k} \notin P_R$. Hence, $n \in N_R$ and $n \notin \langle P_k \rangle$ which shows that $N \nsubseteq P_k$ and thus $M_R =\langle P_k \rangle + N_R$ and $\langle P_k\rangle \neq M_R$ by \cite[Lemma 3.11]{HigueraReyes2023a}. Therefore,  $N_R$ is not a small submodule of $M_R$.
\end{proof}

\begin{remark}\label{Remarkguapo}
   Lemma \ref{AnninLemma2.8} holds if we change the condition that $M[x^{-1}]_R$ is Bass and assume that $R$ is right perfect. By \cite[Exercise 24.7]{Lam1991}, if $R$ is right perfect then every non-zero module $M_R$ has a maximal submodule, and so the proof follows the same arguments.
\end{remark}

Lemma \ref{Lemmaspecialcase} shows that under certain conditions, $M[x^{-1}]_A$ is a hollow module.

\begin{lemma}\label{Lemmaspecialcase} If $M_R$ is a simple module and $Q_A$ is a submodule of $M[x^{-1}]_A$ which contains inverse polynomials of arbitrarily negative degree, then $Q_A=M[x^{-1}]$. In particular, if $M_R$ is simple then $M[x^{-1}]_A$ is hollow.    
\end{lemma}
\begin{proof}
We show by induction on $k \in \mathbb{N}$ that $Q_A$ must contain all inverse monomials of every degree. Let $m(x) \in Q_A$ with leading coefficient $m_k\neq 0$, for some $k \in \mathbb{N}$. Thus $m_k = m(x)x^k \in Q_A$ whence $m_k \in Q \cap M \subseteq M$, and since $M_R$ is simple we have that $Q \cap M = M$. Assume that $Q_A$ contains all monomials of any degree less than $k$. Let $m(x) \in Q_A$ of degree at most $k$ and with leading coefficient $m_l\neq 0$ and $l \ge k$. Then $m(x)x^{l-k} \in Q_A$ with leading term $m'x^{-k}$. By induction hypothesis, all non-leading terms of $m(x)x^{l-k}$ belong to $Q_A$, and so $m'x^{-k} \in Q_A$. It follows that $Q_A$ contains all inverse monomials of any negative degree $k$. 
\end{proof}

Lemma \ref{AnninLemma2.9} is important to prove our main result and extends \cite[Lemma 2.9]{Annin2005}.

\begin{lemma}
\label{AnninLemma2.9} If $M[x^{-1}]_R$ is a right Bass module, then $M_R$ is hollow if and only if $M[x^{-1}]_A$ is hollow.
\end{lemma}
\begin{proof}
Suppose that $M_R$ is hollow and $M[x^{-1}]_A = N_A + N_A'$ for some proper submodules $N_A, N_A'$ of $M[x^{-1}]_A$. If $R$ is right perfect, then there exists a maximal submodule $Q_R$ of $M_R$, and since $M_R$ is hollow, we get that $Q \subseteq_s M$. Lemma \ref{AnninLemma2.8} implies that $Q[x^{-1}] \subseteq_s M[x^{-1}]$, and thus $Q[x^{-1}] + N$ and $Q[x^{-1}] + N'$ are both proper submodules of $M[x^{-1}]_A$ where
\[(Q[x^{-1}] + N)+(Q[x^{-1}] + N')=M[x^{-1}].\]
 Since $Q[x^{-1}]_A$ is a small submodule of $M[x^{-1}]_A$, it follows that $N_A, N_A' \subsetneq Q[x^{-1}]_A$. So, the images of these two modules in $ M[x^{-1}]/Q[x^{-1}] \cong (M/Q)[x^{-1}]$ are non-zero and proper, and they sum to the whole module $(M/Q)[x^{-1}]_A$, that is, $(M/Q)[x^{-1}]_A$ is not hollow. On the other hand, if $(M/Q)_R$ is simple then $(M/Q)[x^{-1}]_A$ is hollow by Lemma \ref{Lemmaspecialcase}, which is a contradiction. Hence, $M[x^{-1}]_A$ is a hollow module.
\end{proof}

Lemma \ref{AnninLemma2.9} is also true if $R$ is right perfect (Remark \ref{Remarkguapo}). The following theorem shows that the couniform dimensions of $M_R$ and $M[x^{-1}]_A$ are equal and generalizes \cite[Theorem 2.10]{Annin2005}.

\begin{Theorem}\label{Annintheorem2.10}
If $M[x^{-1}]_R$ is a right Bass module then %${\rm corank}(M[x^{-1}]_A)={\rm corank}(M_R)$.
\[
{\rm corank}(M[x^{-1}]_A)={\rm corank}(M_R).
\]
\end{Theorem}
\begin{proof}
By Proposition \ref{Proposition1.5Couniform}, if ${\rm corank}(M_R) = n$ then there exist $H_1, \ldots, H_n$ hollow $R$-modules and a surjective homomorphism $\varphi$ of $M_R$ over $\overline{M}_R:=H_1 \oplus \cdots \oplus H_n$ such that ${\rm ker}\ \varphi \subseteq_s M$. In addition, the map $\varphi$ induces a homomorphism $\psi$ of $M[x^{-1}]_A$ over $\overline{M[x^{-1}]}_A:= H_1[x^{-1}]\oplus\cdots\oplus H_n[x^{-1}]$ given by $\psi(mx^{-k}):= \varphi(m)x^{-k}$ for all $mx^{-k}\in M[x^{-1}]_A$. If $m'x^{-i}\in \overline{M[x^{-1}]}_A$ and $\varphi$ is a surjective map, then there exists $m\in M_R$ such that $\varphi(m)=m'$ and thus $\psi(mx^{-i})=\varphi(m)x^{-i}=m'x^{-i}$. Therefore, we have that $\psi$ is surjective. 

Notice that if $mx^{-i} \in {\rm ker}\ \psi$ then $\varphi(m)x^{-i}=0$ which means that $\varphi(m)=0$, and thus $\varphi(m)x^{-i} \in ({\rm ker\ \varphi})[x^{-1}]$. It is not difficult to show that $({\rm ker\ \varphi})[x^{-1}] \subseteq {\rm ker}\ \psi$, and hence ${\rm ker}\ \psi = ({\rm ker\ \varphi})[x^{-1}]$. By Lemma \ref{AnninLemma2.8}, ${\rm ker}\ \psi = ({\rm ker\ \varphi})[x^{-1}] \subseteq_s M[x^{-1}]$, and by Proposition \ref{proposition1.3couniform} (2) we have that
\[{\rm corank}(M[x^{-1}]_A)={\rm corank} \left((M[x^{-1}]/{\rm ker}\ \psi)_A \right) = {\rm corank}\left(\overline{M[x^{-1}]}_A \right).\]
 
If $H_i$ is hollow then $H_i[x^{-1}]_A$ is a hollow module, and thus ${\rm corank} (H_i[x^{-1}]_A)=1$ for all $1 \le i \le n$ by Lemma \ref{AnninLemma2.9}. This implies the equalities
\[
{\rm corank}(M[x^{-1}]_A)= \sum_{i=1}^{n}{\rm corank}\left(H_i[x^{-1}]_A \right)=n.
\]

If ${\rm corank}(M_R)=\infty$, then there exists a surjective homomorphism $\varphi_k$ of $M_R$ over $\overline{M}_R:=N_1 \oplus \cdots \oplus N_k$ with $N_i \neq 0$ and $k\in \mathbb{N}$ arbitrarily large. This homomorphism $\varphi_k$ induces a surjective homomorphism $\psi_k$ of $M[x^{-1}]_A$ over $\overline{M[x^{-1}]}_A$ for each $k\in \mathbb{N}$, which shows that ${\rm corank}(M[x^{-1}]_A)=\infty$.
\end{proof}

As a consequence of Remark \ref{Remarkguapo}, we have the following corollaries.

\begin{corollary} \label{SkewOreTheorem2.10} If $R$ is right perfect then 
\[
{\rm corank}(M[x^{-1}]_A)={\rm corank}(M_R).
\]
\end{corollary}

\begin{corollary}[{\cite[Theorem 2.10]{Annin2005}}] If $R$ is right perfect then 
\[
{\rm corank}(M[x^{-1}]_S)={\rm corank}(M_R).
\]
\end{corollary}

\section{Examples}
   The relevance of the results presented in the paper is appreciated when we extend their application to algebraic structures that are more general than those considered by Annin \cite{Annin2005}, that is, noncommutative rings which cannot be expressed as skew polynomial rings of endomorphism type. 

\begin{example}
  Let $A$ be the Jordan plane $\mathcal{J}(\Bbbk)$ subject to the relation $yx = xy + y^2$ and $M[y^{-1}]_{A}$ the module defined by (\ref{eqn:(3.1)}) and (\ref{eqn:(3.2)}). If $M_{\Bbbk[x]}$ is a right module such that $M[y^{-1}]_{\Bbbk[x]}$ is completely $(\sigma,\delta)$-compatible, then $M_{\Bbbk[x]}$ and $M[y^{-1}]_{A}$ have the same uniform dimension by Theorem \ref{theoremuniform}. Additionally, if $M[y^{-1}]_{\Bbbk[x]}$ is Bass then ${\rm corank}(M_{\Bbbk[x]})= {\rm corank}(M[y^{-1}]_{A})$ by Theorem \ref{Annintheorem2.10}.
\end{example}

%$\sigma(y):=q^{-1}y$ and $\delta(y):=-q^{-1}$ 

\begin{example}
  Consider $A$ as the {$q$-meromorphic Weyl algebra} $MW_q$ and $M[x^{-1}]_{A}$ defined by the relations (\ref{eqn:(3.1)}) and (\ref{eqn:(3.2)}). If $M[x^{-1}]_{\Bbbk[y]}$ is completely $(\sigma,\delta)$-compatible, then the uniform dimensions of $M_{\Bbbk[y]}$ and $M[x^{-1}]_{A}$ are equals by Theorem \ref{theoremuniform}. If $M[x^{-1}]_{\Bbbk[y]}$ is Bass then $M_{\Bbbk[y]}$ and $M[x^{-1}]_{A}$ have the same couniform dimension by Theorem \ref{Annintheorem2.10}. 
  
  %Thinking about the change of variable presented by Lopes (Example \ref{ExampleskewOre} (4)), $MW_q$ can be interpreted as the quantum plane $\mathbb{C}_q[x,y]$ with $yx=qxy$. In this way, if $M[x^{-1}]_{\mathbb{C}[y]}$ is a completely $\sigma$-compatible module with $\sigma(y):=q^{-1}y$, then the description of the attached prime ideals over $M[x^{-1}]_{A}$ follows from theorems \ref{Annintheorem2.1} and \ref{Annintheorem3.2} or Corollary \ref{CorollaryAnnintheorem3.2}.
\end{example}

\begin{example}
   Let $A$ be the ring of skew Ore polynomials of higher order $Q(0,b,c)$ and $M[y^{-1}]_{A}$ defined by (\ref{eqn:(3.1)}) and (\ref{eqn:(3.2)}). If $M[y^{-1}]_{\Bbbk[x]}$ is completely $(\sigma,\delta)$-compatible then ${\rm rudim}(M_{\Bbbk[x]})= {\rm rudim}(M[y^{-1}]_{A})$ by Theorem \ref{theoremuniform}. If $M[y^{-1}]_{\Bbbk[x]}$ is Bass then ${\rm corank}(M_{\Bbbk[x]})= {\rm corank}(M[y^{-1}]_{A})$ by Theorem \ref{Annintheorem2.10}. In a similar way, we get a description of the uniform and couniform dimension of $M[x^{-1}]_{A}$, where $A=Q(a,b,0)$.
\end{example}

\begin{example}
  Consider $A$ as the skew Ore polynomial of higher order defined by Smits \cite{Smits1968} subject to the commutation rule $xr:=r_1x + r_2x^2 + \cdots$ such that $\delta_1$ is an automorphism of $D$ and $\{\delta_2, \ldots, \delta_k \}$ is a set of left $D$-independient endomorphism. If $M[x^{-1}]_{D}$ is completely $(\sigma,\delta)$-compatible then ${\rm rudim}(M_{D})= {\rm rudim}(M[x^{-1}]_{A})$ by Theorem \ref{theoremuniform}. If $M[x^{-1}]_{D}$ is Bass then $M_{D}$ and $M[x^{-1}]_{A}$ have the same couniform dimension by Theorem \ref{Annintheorem2.10}. If $D$ is right perfect then the equality of the couniform dimensions follows from Corollary \ref{SkewOreTheorem2.10}.
\end{example}

\begin{example} Zhang and Zhang \cite{ZhangZhang2008} defined the {\em double extensions} over a $\Bbbk$-algebra $R$ and presented different families of Artin-Schelter regular algebras of global dimension four. It is possible to find some similarities between the double extensions and two-step iterated skew polynomial rings, nevertheless, there exist no inclusions between the classes of all double extensions and of all length two iterated skew polynomial rings (c.f. \cite{Carvalhoetal2011}). Several authors have studied different relations of double extensions with Poisson, Hopf, Koszul and Calabi-Yau algebra (see \cite{RamirezReyes2024} and reference therein). We start by recalling the definition of a double extension in the sense of Zhang and Zhang, and since some typos occurred in their papers \cite[p. 2674]{ZhangZhang2008} and \cite[p. 379]{ZhangZhang2009} concerning the relations that the data of a double extension must satisfy, we follow the corrections presented by Carvalho et al. \cite{Carvalhoetal2011}.

\begin{definition} [{\cite[Definition 1.3]{ZhangZhang2008}; \cite[Definition 1.1]{Carvalhoetal2011}}]\label{DefinitionDoubleExtension}
If $B$ is a $\Bbbk$-algebra and $R$ is a subalgebra of $B$, then 
        \begin{itemize}
            \item[(a)] $B$ is called a {\em right double extension} of $R$ if the following conditions hold: 
        \begin{itemize}
        \item[\rm (i)] $B$ is generated by $R$ and two new variables $y_1$ and $y_2$.
        \item[\rm (ii)] $y_1$ and $y_2$ satisfy the relation
        \begin{equation}\label{eqn:aii}
        y_2y_1 = p_{12}y_1y_2 + p_{11}y_1^2 + \tau_1y_1 + \tau_2y_2 + \tau_0, 
        \end{equation}
        where $p_{12}, p_{11} \in \Bbbk$ and $\tau_1, \tau_2, \tau_0 \in R$.
        \item[\rm (iii)] $B$ is a free left $R$-module with a basis $\left \{y_1^{i}y_2^{j} \ | \ i,j \geq 0 \right \}$.
        \item[\rm (iv)] $y_1R + y_2R + R\subseteq Ry_1 + Ry_2 + R$.
      \end{itemize}
       \item[(b)] A right double extension $B$ of $R$ is called a {\em double extension} if
        \begin{itemize}
        \item[\rm (i)] $p_{12}\neq 0$.
        \item[\rm (iii)] $B$ is a free right $R$-module with a basis $\left \{y_2^{i}y_i^{j} \ | \ i,j \geq 0 \right \}$.
        \item[\rm (iv)] $y_1R + y_2R + R = Ry_1 + Ry_2 + R$.
      \end{itemize}
   \end{itemize}
   \end{definition}
    Condition (a)(iv) from Definition \ref{DefinitionDoubleExtension} is equivalent to the existence of two maps
\begin{center}
    $\sigma(r):=\begin{pmatrix} \sigma_{11}(r) & \sigma_{12}(r) \\ \sigma_{21}(r) & \sigma_{22}(r) \end{pmatrix}
    \ \text{and}\ \delta(r):= \begin{pmatrix} \delta_1(r)\\ \delta_2(r) \end{pmatrix}$ for all $r \in R$,
\end{center}

such that 
\begin{equation}\label{eqn:R2}
    \begin{pmatrix} y_1\\ y_2 \end{pmatrix}r:= \begin{pmatrix} y_1r\\ y_2r \end{pmatrix}=\begin{pmatrix} \sigma_{11}(r) & \sigma_{12}(r) \\ \sigma_{21}(r) & \sigma_{22}(r) \end{pmatrix} \begin{pmatrix} y_1\\ y_2 \end{pmatrix} + \begin{pmatrix} \delta_1(r)\\ \delta_2(r) \end{pmatrix}.
\end{equation}

If $B$ is a right double extension of $R$ then we write $B := R_P [y_1, y_2; \sigma, \delta, \tau]$ where $P:=\{ p_{12}, p_{11} \}\subseteq \Bbbk$, $\tau:=\{\tau_1, \tau_2, \tau_0 \}\subseteq R$ and $\sigma,\delta$ are as above. the set $P$ is called a {\em parameter} and $\tau$ a {\em tail}. If $\delta:=0$ and $\tau$ consists of zero elements then the double extension is denoted by $R_P [y_1, y_2; \sigma]$ and is called a {\em trimmed double extension} \cite[Convention 1.6 (c)]{ZhangZhang2008}. It is straightforward to see that the relation (\ref{eqn:aii}) is given by 
\begin{equation}\label{eqn:RRR}
    y_2y_1 = p_{12}y_1y_2 + p_{11}y_1^2.
\end{equation}
Since $p_{12}, p_{11}\in \Bbbk$ the expression (\ref{eqn:RRR}) can be written as $y_1y_2= p_{12}^{-1}y_2y_1 - p_{12}^{-1}p_{11}y_1^2$. It is clear that $\sigma(y_2)=p_{12}^{-1}y_2$ is an automorphism of $\Bbbk[y_2]$ and $\delta(y_2)=- p_{12}^{-1}p_{11}$ is a locally nilpotent $\sigma$-derivation of $\Bbbk[y_2]$. Thus, the algebra $R_P [y_1, y_2; \sigma]$ can be seen as $A=\Bbbk[y_2](y_1;\sigma,\delta)$. If $M[y_1^{-1}]_{\Bbbk[y_2]}$ is completely $(\sigma,\delta)$-compatible, then $M_{\Bbbk[y_2]}$ and $M[y_1^{-1}]_{A}$ have the same uniform dimension by Theorem \ref{theoremuniform}. If $M[y_1^{-1}]_{\Bbbk[y_2]}$ is Bass then ${\rm corank}(M_{\Bbbk[y_2]})= {\rm corank}(M[y_1^{-1}]_{A})$ by Theorem \ref{Annintheorem2.10}. 

%If $M[y_2^{-1}]_{A}$ is a right module under the action given by  (\ref{eqn:(3.1)}) and (\ref{eqn:(3.2)}) and $M_{\Bbbk[y_2]}$ is a module such that $M[y_1^{-1}]_{\Bbbk[y_2]}$ is completely $(\sigma,\delta)$-compatible, then theorems \ref{Annintheorem2.1} and \ref{Annintheorem3.2} describe the attached prime ideals of $M[y_2^{-1}]_{A}$.
\end{example}

\section{Future work}

Bell and Goodearl \cite{BellGoodearl1988} studied the {\em reduced rank} (or {\em torsionfree rank}) of a particular kind of generalized differential operator ring known as {\em PBW extension}. They proved that if $T$ is a PBW extension over a right Noetherian ring $R$, then $T$ and $R$ have the same reduced rank \cite[Theorem 6.4]{BellGoodearl1988}. Gallego and Lezama \cite{GallegoLezama2011} introduced the {\em skew PBW extensions} as a generalization of the PBW extensions and skew polynomial rings of injective type. Since its introduction, ring and homological properties of skew PBW extensions have been widely studied (see \cite{LFGRSV} and reference therein). In this way, it is natural to investigate the reduced rank of skew PBW extensions.

\end{document}